\documentclass[12pt]{article}

\usepackage{verbatim}
\usepackage{amsmath}
\usepackage{amsthm}
\usepackage{eucal}
\usepackage{amssymb}
\usepackage{epsfig}
\newtheorem{theo}{Theorem}
\newtheorem{prop}[theo]{Proposition}

\newtheorem{lemma}[theo]{Lemma}
\newcommand {\pare}[1] {\left( {#1} \right)}
\newcommand {\cro}[1] {\left[ {#1} \right]}
\newcommand {\acc}[1] {\left\{ {#1} \right\}}
\newcommand {\nor}[1] { \left\| {#1} \right\|}
 
\def \R  {\mathbb{R}} 
\def \T  {\mathbb{T}} 
\def \Z  {\mathbb{Z}}
\def \CC {{\cal C}}
\def \FF {{\cal F}}
\def \ind {\hbox{ 1\hskip -3pt I}}
\newcommand {\va}[1] {\left| {#1} \right|}
\def \card {\mbox{card}}

\begin{document}
\title{\textsc{Large deviations for self-intersection local times of stable random walks}}
\author{Cl\'{e}ment Laurent}
\maketitle
\begin{abstract}
Let $(X_t,t\geq 0)$ be a random walk on $\mathbb{Z}^d$. 
Let   $ l_T(x)=  \int_0^T \delta_x(X_s)ds$ the local time at the state $x$ and
$ I_T= \sum\limits_{x\in\mathbb{Z}^d} l_T(x)^q $ the q-fold self-intersection local time (SILT). In \cite{Castell} Castell proves a large deviations principle for the SILT of the simple random walk in the critical case $q(d-2)=d$. 
In the supercritical case $q(d-2)>d$, Chen and M\"orters obtain in \cite{ChenMorters} a large deviations principle for the intersection of $q$ independent random walks , and Asselah obtains in \cite{Asselah5} a large deviations principle for the SILT with $q=2$. We extend these results to an $\alpha$-stable process (i.e. $\alpha\in]0,2]$) 
in the case where $q(d-\alpha)\geq d$.
\end{abstract}

\section{Introduction}
Let $(X_t,t\geq 0)$ be a continous time random walk on $\mathbb{Z}^d$ with jump rate $1$, whose generator is denoted $A$: 
$$Af(x)=\sum\limits_{y\in\mathbb{Z}^d}\mu(y-x)(f(y)-f(x))$$
where $\mu$ is the law of
the increment. We assume that $\mu$ is in the domain of attraction of a stable law of index $\alpha$ and that $\mu$ is symmetric. More precisely we assume the following assumption:
\paragraph{Assumption 1:}
\begin{itemize}
\item
$\exists\ c_1,c_2>0 \text{ such that }\forall x,y\in\mathbb{Z}^d,\ \frac{c_1}{\vert y-x\vert^{d+\alpha}} \leq \mu(y-x) \leq \frac{c_2}{\vert y-x\vert^{d+\alpha}}$.
\item $\mu$ is symmetric.
\end{itemize}
 In this article we are interested in the q-fold self intersection local time (SILT), i.e.:
\[ I_T= \sum\limits_{x\in\mathbb{Z}^d} l_T(x)^q \text{\ \ with\ \ }       l_T(x)=  \int_0^T \delta_x(X_s)ds. \] 

The study of self-intersection is naturally arising from both probability and physics. In probability this quantity naturally arises from study of random walk in random scenery for instance. In physics we can cite the Polaron problem in quantum mechanics and the study of polymers in statistical mechanic. For the latter, represent a polymer as a chain of $N$ molecules which is considered as a random walk $(X_n,n\in [0,N])$. Physicists study measures of the form $\exp(-\beta I_N)$ where $I_N$ is the discrete analogous of $I_T$. When $\beta<0$, the measure favors unfolded polymers with few intersections, whereas when $\beta>0$, the measure favors the self-intersections of the polymers.

To give an idea of the behaviour of $I_T$ to the reader, we focus on the most studied case with $\alpha=2$ and $q=2$, which means that we consider the $l_2$-norm of the local times of a random walk with finite variance. The first idea is to point out the very important role played by the transience or the recurrence of the walk. Of course, when the walk is recurrent (dimension 1 and 2), it will intersect itself much more than when it is transient (dimension $d\geq 3$). Hence the SILT will be much more large. More precisely for $d=1$, $I_T\sim T^{3/2}$; for $d=2$, $I_T\sim T\log(T)$; and for $d\geq 3$, the walk being transient it spends a time of order 1 at each site and $I_T\sim T$.

The difference between recurrence and transience reappears in the central limit theorem. In dimension 1 and 2, we have a convergence to the local time of a Brownian motion (renormalized for $d=2$), while for $d=3$ a convergence to a normal law takes place:
\begin{itemize}
\item $d=1$: $\frac{I_T}{T^{3/2}}\xrightarrow {(d)} \gamma_1$, where $\gamma_1$ is the intersection local time of a Brownian motion.
\item $d=2$: $\frac{I_T-E[I_T]}{T}\xrightarrow {(d)}\gamma_1'$, where $\gamma_1'$ is the renormalized intersection local time of a Brownian motion.
\item $d\geq 3$:  $\frac{I_T-E[I_T]}{\sqrt{\text{Var}(I_T)}}\xrightarrow {(d)} \mathcal{N}(0,1)$.
\end{itemize}

Since the law of large numbers and limit laws have been established, it is natural to be interested in the large deviations of the SILT.

The large deviations are the study of rare events. In this article we  wonder how $I_T$ can exceed its mean, i.e. we compute the probability $P(I_T\geq b_T^q)$ where $b_T^q\gg E[I_T]$. Heuristically, it is interesting to ask how the walk can realize this kind of atypical event. We propose here a classical strategy for the walk to realize large deviations of its SILT.

Let us localize the walk on a ball of radius R up to time $\tau$.
As $\mu$ is in the domain of attraction of a stable law, there exists $(U_t,t\geq 0)$ a non degenerate stable process such that $\frac{1}{a(t)} X_t \longrightarrow U_1$, where $a(t)\sim t^\frac{1}{\alpha}$. On one hand, the walk arrives at the edge of the ball in $R^\alpha$ units of time and the probability of this localization is about $\exp (-\frac{\tau}{R^\alpha})$. 
On the other hand, the walk spends  about $\frac{\tau}{R^d}$ units of time on each site of the ball, so $I_T$ increases to $\pare{\frac{\tau}{R^d}}^q R^d=\tau^q R^{d(1-q)}$.
We want $I_T=b_T^q$, which gives $\tau=b_T R^\frac{d(q-1)}{q}$. Thus the probability of this localization is about $\exp\pare{-b_T R^{\frac{d(q-1)}{q}-\alpha}}$. Maximizing this quantity in $R$, we obtain three cases:
\begin{itemize}
\item $\frac{d(q-1)}{q}-\alpha>0\Leftrightarrow q(d-\alpha)>d$ (supercritical case): in this case the optimal choice for $R$ is $1$. A good strategy to realize the large deviations is to spend a time of order $b_T$ in a ball of radius $1$, and then: $P(I_T\geq b_T^q)\sim \exp(-b_T)$.
\item $\frac{d(q-1)}{q}-\alpha=0\Leftrightarrow q(d-\alpha)=d$ (critical case): here the choice of $R$ does not matter. Every strategy consisting in spending a time of order $b_T R^\frac{d(q-1)}{q}$ in a ball of radius $R$ such that $1\leq R\ll (T/b_T)^{1/\alpha}$ is a good strategy, so $P(I_T\geq b_T^q)\sim \exp(-b_T)$.
\item $\frac{d(q-1)}{q}-\alpha<0\Leftrightarrow q(d-\alpha)<d$ (subcritical case): a good strategy is to stay up to time $T$ in a ball of maximal radius, i.e. $\pare{\frac{T}{b_T}}^\frac{q}{d(q-1)}$, thus $P(I_T\geq b_T^q)\sim\exp\pare{-b_T\pare{\frac{b_T}{T}}^{\frac{\alpha q}{d(q-1)}-1}}$. 
\end{itemize}

The question of large deviations for the SILT of random walk has very studied in recent years. The knowledge of the case $\alpha =2$ is the most progressed. We make here a brief review of these results.
\begin{itemize}
\item
For $d=1$, Chen and Li obtain a large deviations principle in \cite{BassLi}, as they obtain similar results for Brownian motion.
\item 
For a large deviations principle in the case $d=2$, we refer to the work of Bass, Chen and Rosen \cite{BCR2}. They express the constant in term of the best possible constant in a Gagliardo-Nirenberg inequality.
\item
In \cite{Chen}, Chen obtains a large deviations principle for all the scales of  deviations for $d=3$ and $q=2$. For dimension 2 and 3, the main idea is to first establish the large deviations of $q$ independent random walk then to use the dyadic decomposition due to Westwater \cite{West}.
\item 
In the critical dimension $d=4$, a recent paper of Castell \cite{Castell} states a large deviations principle, the constant being given in term of the best possible constant in a Gagliardo-Nirenberg inequality. 
\item 
The case of the supercritical dimension $d\geq 5$ is treated in two papers. In \cite{ChenMorters}, Chen and M\"orters give a large deviations principle concerning mutual intersection local times of $q$ independent random walks in infinite time horizon, and Asselah obtains in \cite{Asselah5} a large deviations principle for the SILT of a symmetric random walk. The method used by Castell in \cite{Castell} and by Chen and M\"orters in \cite{ChenMorters} have the same idea at their core. Indeed, Chen and Rosen explicitely compute large moments of the SILT and Castell uses Einsenbaum's Theorem, whose proof is based on the computation of its large moments.
\end{itemize}
A recent book of Chen \cite{Chen} summarizes these results. We refer the interested reader to this work for a precise development of the subject.

In this article, we are interested in the case where $\alpha <2$, i.e. the $\alpha$-stable random walk.
Up to now only subcritical case $q(d-\alpha)<d$ is solved in three papers, \cite{BCR},\cite{ChenLiRosen} and \cite{ChenRosen}. In these three articles the authors obtain some large deviations principle, and give the constant in terms of the best possible constant in a Gagliardo-Nirenberg inequality. We briefly present these results.
\begin{itemize}
\item The case $\alpha >d$ (note that imply $d=1$) is solved by
Chen, Li and Rosen in \cite{ChenLiRosen}. They obtain a large deviations principle for the SILT.

\item The case $\alpha\leq d$ is studied in two articles. Bass, Chen and Rosen explore the specific case
$p=2$ and $\alpha\in]\frac{2d}{3},d]$ in \cite{BCR}. They show a large deviation principle for the SILT.
The idea of the proof is to first study the intersection of two independent random processes, then to use the dyadic decomposition due to Westwater.
\item To complete the picture in the case $q(d-\alpha)<d$, Chen and Rosen \cite{ChenRosen} obtain a large deviations principle for intersection of $q$ independent stable processes using Feynman-Kac type large deviations.
\end{itemize}

  This article contributes to the question of large deviations for the self-intersection local times. We get a large deviations principle in the critical and supercritical case (i.e. $q(d-\alpha)\geq d$). In this situation the local times of the $\alpha$-stable process do not exist and we have to consider the SILT of the random walk itself. We point out that our method allows us to consider the q-fold self intersection local times even if $q$ is a real number instead of $q$ is an integer. Moreover, denote by $Q_T$ the mutual intersection of $q$ independent random walks $(X_t^{(i)}, t\geq 0,1\leq i\leq q)$ defined by
  $$Q_T=\sum\limits_{x\in\mathbb{Z}^d}\prod_{i=1}^ql_T^{(i)}(x)=\int_0^T\cdot\cdot\cdot\int_0^T \ind_{X_{s_1}^{(1)}=\cdot\cdot\cdot =X_{s_q}^{(q)}}ds_1\cdot\cdot\cdot ds_q.$$
The upper bound of the large deviations principle for the SILT leads to an upper bound of large deviations for $Q_T$ by the following inequality:
$$Q_T^{1/q}=\pare{\sum\limits_{x\in\mathbb{Z}^d}\prod_{i=1}^ql_T^{(i)}(x)}^{1/q}\leq \pare{\prod_{i=1}^q\nor{l_T^{(i)}}_q}^{1/q}\leq \frac{1}{q}\sum_{i=1}^q\nor{l_T^{(i)}}_q.$$

As $q(d-\alpha)\geq d$ we have $\alpha <d$, which implies that the walk is transient. So $l_T(x)\sim 1$ and $I_T\sim T$ but of course $I_T\leq T^q$. Therefore we focus on the probability  $P(I_T\geq b_T^q)$ for $T\gg b_T\gg T^\frac{1}{q}$.\\
 \paragraph{Main results} $ $

 Let $G$ be the Green function of the random walk $(X_t,t\geq 0)$. Remark that as we have $\alpha <d$, the walk is transient, which gives that the Green function does exist. We use the following notations:
 
\begin{align*}
&\rho (q)= \sup\limits_{g} \acc {<g,Gg>, \text{ supp(g) compact, } \nor{g}_{(2q)'}=1}, \\
&\kappa (q)=\inf_f \acc{ 
\frac{<f,-Af>}{\nor{f}_{2q}^2}, \ \nor{f}_2=1} 
\end{align*}
where $<\cdot , \cdot >$ is the classical scalar product on $l^2(\mathbb{Z}^d)$ and $Gg(x)=\sum\limits_{y\in\mathbb{Z}^d}G(x-y)g(y)$.

\begin{prop}\label{kapparho}
Under assumption 1, if $q(d-\alpha)\geq d$, then $\kappa(q)$ is a non-degenerate constant and $\kappa(q)=\frac{1}{\rho(q)}$.
\end{prop}

\begin{theo} {\bf Large deviations.} 
\label{ld}

 Assume that $q(d-\alpha)\geq d$ and $T \gg b_T \gg T^\frac{1}{q}$. Under assumption 1, we have:
\begin{equation} 
\label{ld.eq}
 \lim_{T \rightarrow \infty} \frac{1}{b_T} \log P \cro{I_T \geq b_T^q}
=  - \frac{1}{\rho (q)} \, .
\end{equation}
\end{theo}

\paragraph{Sketch of the proof}$ $

The proof of the lower bound of large deviations (Theorem \ref{tlb}) is classical. Let $\mathcal{F}$ be the set of the probability measures on $\mathbb{Z}^d$ endowed by the weak topology of probability measures. Donsker and Varadhan have proved a restricted large deviation principle for $\frac{l_T}{T}$ in $\mathcal{F}$ with rate function $\mathcal{J}(\nu)=<\sqrt{\nu},-A\sqrt{\nu}>$. Then the lower bound of the large deviations with constant $\kappa(q)$ follows from the lower semicontinuity of the function   
$$\nu \in \FF \mapsto \nor{\nu}_q
=\sup\limits_{f; \nor{f}_{q'}=1} \acc{\sum\limits_{x\in\mathbb{Z}^d} \nu(x) f(x)}.$$ 

However the large deviations principle for $\frac{l_T}{T}$ being restricted, that is the upper bound is only true on compact sets, we cannot use it for the upper bound. The method used here for the upper bound has been recently developed by Castell in \cite{Castell}. The main idea is to use Eisenbaum's theorem to shift the problem from a symmetric Markov process to a Gaussian process, which is considerably more convenient. Indeed, this theorem relates the law of the local times of a symmetric Markov process stopped at an exponential time with the square of a Gaussian process, whose covariance is given by the Green kernel of the stopped Markov process.

First we compare the SILT of the random walk with the SILT of the random walk projected on the torus, and stopped at an exponential time of parameter $\lambda$ independent of the walk (lemma \ref{majoesp}). Then we apply Eisenbaum's theorem (theorem \ref{eisenbaum.theo}) to arrive at the Gaussian process $(Z_x,x\in\T_R)$ whose covariance is given by $G_{R,\lambda}(x,y)=E_x \cro{\int_0^\tau \delta_y(X^{(R)}_s) \, ds}$ (lemma \ref{firstbornesup} and \ref{ldfz.lem}).  In lemmas \ref{firstbornesup} and \ref{ldfz.lem} we work on the Gaussian process $(Z_x,x\in\T_R)$ using concentration inequalites for norms of Gaussian processes. We let space and time going together to infinity to obtain a first upper bound with a constant $-1/\rho_1$.

We finish the proof of the upper bound by proving in proposition \ref{rho1} that $\rho_1\leq \rho(q)$. The estimates of the transition probability of an $\alpha$-stable random walk obtained by Bass and Levin in \cite{BassLevin} are a key of its proof. We assume assumption 1 because Bass and Levin need it to obtain these estimates. The upper bound in this assumption is not surprising since the increments of the walk have moments of order $\alpha$. However, the lower bound is less natural since it imposes the walk to jump of arbitrary distance in $\mathbb{Z}^d$. Current results concerning estimates of transition probabilities for $\alpha$-stable processes require this kind of assumption. We think that this assumption is not necessary to obtain large deviations of the SILT, and it would be interesting to do without it. 

 Letting $R$ and $T$ go to infinity  together ask the question of scale between $\lambda$, $R$ and $T$. As we stop the random walk at an exponential time $\tau$ of parameter $\lambda$, we must control the quantity $\frac{1}{b_T}\log P(\tau \geq T)=\frac{\lambda T}{b_T}$. That's why we define $\lambda$ as $a\frac{b_T}{T}$.
Additionally, the Eisenbaum's theorem shift the problem from the $l_q$-norm of the local time $l_T$ to the $l_{2q,R}$-norm of the Gaussian process $(Z_x,x\in\T_R)$. Since $\nor{Z}_{2q,R}^2\sim R^{d/q}$ we have the extra constraint  $b_T\geq R^{d/q}$. Those two precedent conditions, combined with the condition $\lambda R^{d/q'}\gg 1$ coming from proposition \ref{rho1}, imply that $b_T^q\gg T$. That's why the proof does not work at the scale of the mean $b_T\sim T^{1/q}$.

 Next, we have to equalize the lower and upper bound, which is equivalent to prove $\kappa(q)=1/\rho(q)$. This is done in proposition \ref{kappa} where we use some techniques of Chen and M\"orters from \cite{ChenMorters}.

Finally it remains to prove that our constants $\kappa(q)$ and $\rho(q)$ are not degenerate, which is done in proposition \ref{kappafini}. We want to point out that in the supercritical case it is not difficult to prove that $\rho(q)$ is finite. Indeed, from the results of Le Gall and Rosen \cite{LeGallRosen}, we know that
$G(0,x)=O(\vert x\vert ^{\alpha -d})$, which implies that $\nor{G}_q$ is finite in the supercritical case $q(d-\alpha)>d$. These estimates cannot answer the question when $q(d-\alpha)=d$. So we had to work on $\kappa(q)$ and the underlying Sobolev's inequalities. The solution comes on one hand, from a work of Varopoulos \cite{Varopoulos} which relates Sobolev's inequalities and estimates of the probability transition, and on the other hand,  from estimates of the probability transition obtained thanks to the work of Bass and Levin \cite{BassLevin}.

This article is organized as follows. Section 2 is devoted to the proofs of two preliminary lemmas, giving some informations on the Green function. We prove a  first upper bound in section 3 and give in section 4 the demonstration of the lower bound. Finally in section 5 we end the proof of the upper bound by proving that the constant is not degenerate and equalizing the bounds.

\section{Around the Green function}
In this section we prove some preliminary results about the Green function which will be used throughout this article.

Set $G_{R,\lambda}$ the Green function of the walk $(X_t,t\geq 0)$ projected on the torus $\T_R$ and stopped at an exponential time $\tau$ of parameter $\lambda$ independent of the random walk. We use the same notation $x$ for $x\in\T_R$ and for its representant in $[0,R[^d$.

\begin{lemma}
\label{ineqgreen}
Under assumption 1, there exists a constant $C$ such that $\forall \lambda,R>0$: 
\begin{equation}
G_{R,\lambda}(x,y)\leq G(x,y)+\frac{C}{\lambda R^d}.
\end{equation}
\end{lemma} 

\begin{proof}

Let $p^R_t(x,y)$ be the transition probability of $X^{(R)}_t$ the random walk $X_t$ projected on the torus $\T_R$, hence
\begin{equation*}
G_{R,\lambda}(x,y) = \mathbb{E}_x \cro{\int_0^\tau \ind_{X_s^R =y}ds}
 					=\int_0^{+\infty} \exp(-\lambda t) p^R_t(x,y) dt. 
\end{equation*}

By Theorem 1.1 in \cite{BassLevin}, there exists a constant $C$ such that
\begin{equation*}
\forall t\geq 0,\forall x,y\in\mathbb{Z}^d,p_t(x,y)\leq C\sum\limits_{z\in\mathbb{Z}^d}\pare{t^{-d/\alpha}\wedge \frac{t}{\vert x-y\vert^{d+\alpha}}}.
\end{equation*}
Using the change of variable $\xi =z-\frac{x-y}{R}$ we have:
\begin{align*}
 p_t^R(x,y)&=\sum\limits_{z\in \mathbb{Z}^d} p_t(x,y+Rz)\\
 &\leq p_t(x,y)+C\sum\limits_{\underset{\vert z\vert\leq \frac{t^{1/\alpha}}{R}}{z\neq 0}}\frac{1}{t^{d/\alpha}} +C\sum\limits_{\vert z\vert > \frac{t^{1/\alpha}}{R}}\frac{t}{(R\vert z\vert)^{d+\alpha}}\\
& \leq p_t(x,y)+\frac{C}{R^d}+C\sum\limits_{\vert z\vert > \frac{t^{1/\alpha}}{R}}\frac{t}{(R\vert z\vert)^{d+\alpha}}.
 \end{align*}
 Consequently for $L>1$ we have:
\begin{align}
 \nonumber G_{R,\lambda}(x,y)
 \leq &\int_0^{+\infty}\exp(-\lambda t)p_t(x,y)dt+\int_0^{+\infty} \exp(-\lambda t)\pare{\frac{C}{R^d}+C\sum\limits_{\vert z\vert > \frac{t^{1/\alpha}}{R}}\frac{t}{(R\vert z\vert)^{d+\alpha}}}dt\\
 \nonumber \leq &G(x,y)+\frac{C}{\lambda R^d} +C \int_0^{+\infty}\frac{t \exp(-\lambda t)}{R^{d+\alpha}} \sum\limits_{\vert z\vert > \frac{t^{1/\alpha}}{R}}\frac{1}{\vert z\vert^{d+\alpha}}dt\\
\nonumber =&G(x,y)+\frac{C}{\lambda R^d}+C\int_0^{LR^\alpha}\frac{t \exp(-\lambda t)}{R^{d+\alpha}} \sum\limits_{\vert z\vert > \frac{t^{1/\alpha}}{R}}\frac{1}{\vert z\vert^{d+\alpha}}dt\\
\label{aa} &+C\int_{LR^\alpha}^{+\infty}\frac{t \exp(-\lambda t)}{R^{d+\alpha}} \sum\limits_{\vert z\vert > \frac{t^{1/\alpha}}{R}}\frac{1}{\vert z\vert^{d+\alpha}}dt.
 \end{align} 
Let us find an upper bound for the first integral in (\ref{aa}). Using the fact that the function $x\rightarrow \frac{1-\exp(-x)}{x}$ is bounded on $\mathbb{R}^+$ we obtain:
\begin{align}
\nonumber
\int_0^{LR^\alpha}\frac{t \exp(-\lambda t)}{R^{d+\alpha}} \sum\limits_{\vert z\vert > \frac{t^{1/\alpha}}{R}}\frac{1}{\vert z\vert^{d+\alpha}}dt 
&\leq \sum\limits_{\vert z\vert > 0}\frac{1}{\vert z\vert^{d+\alpha}}\int_0^{LR^\alpha}\frac{t \exp(-\lambda t)}{R^{d+\alpha}} dt\\
 &\leq C \frac{1-\exp(-\lambda LR^\alpha)}{\lambda^2 R^{d+\alpha}}
\label{cc}\leq \frac{C}{\lambda R^d}.
\end{align}
We work now on the second integral in (\ref{aa}):
\begin{align*}
\int_{LR^\alpha}^{+\infty}\frac{t \exp(-\lambda t)}{R^{d+\alpha}} \sum\limits_{\vert z\vert > \frac{t^{1/\alpha}}{R}}\frac{1}{\vert z\vert^{d+\alpha}}dt
&\leq C\int_{LR^\alpha}^{+\infty}\frac{t \exp(-\lambda t)}{R^{d+\alpha}} \sum\limits_{k > \frac{t^{1/\alpha}}{R}}\frac{1}{k^{1+\alpha}}dt\\
&\leq C\int_{LR^\alpha}^{+\infty}\frac{t \exp(-\lambda t)}{R^{d+\alpha}}
\frac{1}{\pare{\frac{t^{1/\alpha}}{R}-1}^\alpha}dt\\
&= C\int_{LR^\alpha}^{+\infty} \frac{t\exp(-\lambda t)}{R^{d}}\frac{1}{\pare{t^{1/\alpha}-R}^\alpha}dt.
\end{align*}
Since $t\geq LR^\alpha$, we have $t^{1/\alpha}-R\geq t^{1/\alpha}(1-L^{-1/\alpha})$, then:
\begin{align}
\nonumber \int_{LR^\alpha}^{+\infty}\frac{t \exp(-\lambda t)}{R^{d+\alpha}} \sum\limits_{\vert z\vert > \frac{t^{1/\alpha}}{R}}\frac{1}{\vert z\vert^{d+\alpha}}dt
&\leq C\int_{LR^\alpha}^{+\infty} \frac{\exp(-\lambda t)}{R^d(1-L^{-1/\alpha})^\alpha}dt\\
\label{cd}&\leq C\frac{\exp(-L\lambda R^\alpha)}{\lambda R^d (1-L^{-1/\alpha})^\alpha}
\leq \frac{C}{\lambda R^d}.
\end{align}
Gathering (\ref{aa}),(\ref{cc}) and (\ref{cd}) we obtain:
\begin{align*}
G_{R,\lambda}(x,y)\leq G(x,y) + \frac{C}{\lambda R^d}.
\end{align*}
\end{proof}

\begin{lemma}
\label{green}
Assume that $\lambda$ and $R$ depend on $T$ in such a way that $\lambda \ll 1$ and $\lambda R^d \gg 1$. Under assumption 1, we have:
\begin{equation}
\lim\limits_{T\rightarrow +\infty}G_{R,\lambda}(0,0) =G(0,0).
\end{equation}
\end{lemma}

\begin{proof}
On one hand, by lemma \ref{ineqgreen} there exists a constant $C$ such that $\forall \lambda,R>0$, $G_{R,\lambda}(0,0)\leq G(0,0)+\frac{C}{\lambda R^d}$. Hence with $\lambda R^d\gg 1$ we have: 
\begin{equation*}
\limsup\limits_{T\rightarrow +\infty}G_{R,\lambda}(0,0) \leq G(0,0).
\end{equation*}
On the other hand let $S>0$. Using the fact that $1\geq p^R_t(0,0) \geq p_t(0,0)$ and $\exp(-\lambda t)\leq 1$ for $t\geq 0$, we deduce:
\begin{align*} 
\nonumber\int_0^{S} \exp(-\lambda t) p^R_t(0,0) \, dt 
&= \int_0^S p_t^R(0,0)+(\exp (-\lambda t)-1)p_t^R(0,0)dt\\
& \nonumber\geq \int_0^{S}  p_t(0,0) \, dt  -  \int_0^{S} (1-\exp(-\lambda t)) \, dt 
\\
& = \label{s} \int_0^{S}  p_t(0,0) \, dt  + \frac{\exp(-\lambda S)-1+\lambda S}{\lambda}.
\end{align*} 
If $S$ is chosen so that 
$S \gg 1$ and 
$\frac{1}{\lambda}\pare{\exp(-\lambda S)-1+\lambda S}\ll 1$,
then we have:
\begin{equation}
\label{u}
\liminf_{T \rightarrow \infty} G_{R,\lambda}(0,0)\geq\liminf_{T \rightarrow \infty}  \int_0^{S} \exp(-\lambda t) p^R_t(0,0) \, dt
\geq \int_0^{\infty} p_t(0,0) \, dt := G(0,0) . 
\end{equation}

Using Taylor series conditions, $S\gg 1$ and $\lambda S^2 \ll 1$
are sufficient. These conditions are compatible because $\lambda\rightarrow 0$. So, for a such choice of $S$, we have:
\[ \liminf_{T \rightarrow \infty} G_{R,\lambda}(0,0) 
\geq  G(0,0).
\]
\end{proof}

\section{Upper bound}
In this section we obtain a first upper bound for the large deviations of $I_T$ which is given in theorem \ref{fpgd}.

\begin{theo}
\label{fpgd} 
Assume that $q(d-\alpha)\geq d$ and that we are under assumption 1.
For all $a > 0$ we define the parameter $\lambda$ of the exponential time $\tau$ by $\lambda=a\frac{b_T}{T}$.
Moreover, assume that $\lambda$, $R$ and $b_T$ depend on $T$ in such a way that $\lambda R^d\gg 1$, $b_T\gg R^{d/q}$ and $\log(T) \ll b_T \ll T$. Then we define
\[
\rho_1(a) = \limsup_{T \rightarrow \infty} \rho_1(a,R,T)
\ and\ \rho_1 = \limsup_{a \rightarrow 0} \rho_1(a) , 
\]
where $\rho_1(a,R,T):= \sup
		\acc{\sum\limits_{x,y\in \T_R} f(x) G_{R,\lambda}(x,y) f(y) \,
; \, f \text{ such that }  \nor{f}_{(2q)',R}=1}$, and we have:
$$\limsup\limits_{T\rightarrow +\infty}\frac{1}{b_T} \log P\cro{I_T\geq b_T^q}\leq -\frac{1}{\rho_1}.$$
\end{theo}

The method of the proof is similar to the one developed by Castell in \cite{Castell}. We give it for the sake of completeness. 
\subsection{Step 1: comparison with the SILT of the random walk on the torus stopped at an exponential
time}

\begin{lemma}
\label{majoesp}
Let $\tau$ be the exponential time defined in theorem \ref{fpgd}.
Let $l^{(R)}_{\tau}(x)=\int_0^{\tau} 
\delta_x (X_s^{(R)}) \, ds$ and $I_{R,\tau}= \sum_{x \in \T_R} (l^{(R)}_{\tau}(x))^q$.
Then $\forall a,R,T >0$: 

\begin{equation*}
P \cro{ I_T\geq b_T^q}
\leq e^{a b_T} P \cro{ I_{R,\tau}\geq b_T^q}. 
\end{equation*}
 \end{lemma}

\begin{proof}
We deduce by convexity that
\begin{align*}
I_T  &= \sum_{x \in \Z^d} l_T^q(x) 
	 = \sum_{x \in \T_R} \sum_{k \in \Z^d} l^q_T(x+kR)\\
    & \leq  \sum_{x \in \T_R} \pare{\sum_{k \in \Z^d} l_T(x+kR)}^q
      =\sum_{x \in \T_R} l^q_{R,T}(x) = I_{R,T}.
\end{align*}

Then using the fact that $\tau\sim\epsilon (\lambda)$ independent of $(X_s,s\geq 0)$ with $\lambda=a\frac{b_T}{T}$, we get:
\begin{align*}
P \cro{ I_T\geq b_T^q} \exp\pare{-a b_T}
		&\leq P \cro{ I_{R,T}\geq b_T^q} P(\tau\geq T)\\
		&\leq \mathbb{P} \cro{ I_{R,T}\geq b_T^q,\tau\geq T}\\
		&\leq P \cro{ I_{R,\tau}\geq b_T^q}.
\end{align*}
Finally, $P \cro{ I_T\geq b_T^q}
\leq e^{a b_T} P \cro{ I_{R,\tau}\geq b_T^q} $.

\end{proof}

\subsection{Step 2: the Eisenbaum isomorphism theorem}
We use here the following theorem due to Eisenbaum given by corollary 8.1.2 page 364 in \cite{MarcusRosen}.
\begin{theo} (Eisenbaum) 
\label{eisenbaum.theo}
Let $\tau$ be as in theorem \ref{fpgd} and
let $(Z_x, x \in \T_R)$ be a centered Gaussian process 
with covariance matrix $G_{R,\lambda}$ 
independent of $\tau$  and of the  random walk $(X_s, s \geq 0)$. 
For $s\neq 0$, consider the process 
$S_x := l_{R,\tau}(x) + \frac{1}{2} (Z_x+s)^2$. Then for all measurable
and bounded function $F : \R^{\T_R} \mapsto \R$: 
\begin{equation*}
E\cro{F((S_x; x\in \T_R))}
= E\cro{F\pare{(\frac{1}{2}(Z_x +s)^2;x \in \T_R)} \, \pare{1 + \frac{Z_0}{s}}}
\, . 
\end{equation*} 
\end{theo} 

\subsection{Step 3: Comparison between $I_{R,\tau}$ and $\nor{Z}_{2q,R}$}
\begin{lemma}
\label{firstbornesup}
Let $\tau$ and $(Z_x, x\in \T_R)$ be defined
as in theorem \ref{eisenbaum.theo}. 
 $\forall\epsilon >0$, there exists a constant $C(\epsilon) \in ]0;\infty[$
depending only on $\epsilon$ such that $\forall  a,\gamma ,R,T >0$:
\begin{align*} 
P \cro{ I_{R,\tau}\geq b_T^q}
\leq 
C(\epsilon) \exp\pare{-\gamma b_T(1+\circ(\epsilon))}
\pare{1+\frac{R^\frac{d}{2q}\sqrt{ T}}{\epsilon b_T\sqrt{2a\epsilon}}} 
\frac{E\cro{\exp\pare{ \frac{\gamma }{2} \nor{Z}_{2q,R}^2}}^{\frac{1}{1+\epsilon}}}
{P\cro{ \nor{Z}_{2q,R} \geq 2 \sqrt{2b_T\epsilon}}}
\end{align*}  
where $\nor{\cdot}_{2q,R}$ is the
$l_{2q}$-norm of functions on $\T_R$. 
\end{lemma}

\begin{proof}
\begin{align*}
 S_x :=l_{R,\tau}(x) + \frac{1}{2} (Z_x+s)^2 &
\Rightarrow  S_x^q\geq l^q_{R,\tau}(x) + \pare{\frac{1}{2} (Z_x+s)^2}^q\\
&\Rightarrow  \sum_{x\in\T_R} S_x^q \geq I_{R,\tau} + \sum_{x\in\T_R}\frac{1}{2^q} (Z_x+s)^{2q}.
\end{align*}
By independence of $(Z_x,x\in\T_{R})$ with the random walk $(X_s,s\geq 0)$ and the exponential time $\tau$, we have $\forall \epsilon >0$,
\begin{align}
 \nonumber P\pare{I_{R,\tau}\geq b_T^q } P\pare{\sum_{x\in\T_R}\frac{1}{2^q} (Z_x+s)^{2q}\geq b_T^q\epsilon^q}
=  &\nonumber P\pare{I_{R,\tau}\geq b_T^q , \sum_{x\in\T_R}\frac{1}{2^q} (Z_x+s)^{2q}\geq b_T^q\epsilon^q}\\
\leq & \nonumber P\pare{I_{R,\tau}+\sum_{x\in\T_R}\frac{1}{2^q} (Z_x+s)^{2q}\geq b_T^q (1+\epsilon^q)}\\
= & \nonumber P\pare{\sum_{x\in\T_R} l_{R,\tau}(x)^q + \frac{1}{2^q}(Z_x+s)^{2q}\geq b_T^q (1+\epsilon^q)}\\
\leq &\nonumber P\pare{\sum_{x\in\T_R}S_x^q\geq b_T^q(1+\epsilon^q)}\\
= &\label{ab} E\cro{\pare{1+\frac{Z_0}{s}}\ind_{ \sum\limits_{x\in\T_R}\frac{1}{2^q} (Z_x+s)^{2q}\geq b_T^q(1+\epsilon^q)}},
\end{align}
where the last equality comes from Theorem \ref{eisenbaum.theo}.
Moreover by Markov inequality, $\forall \gamma >0$,

\begin{align}
\nonumber
&E\cro{\pare{1+\frac{Z_0}{s}}\ind_{\sum_{x\in\T_R}\frac{1}{2^q} (Z_x+s)^{2q}\geq b_T^q(1+\epsilon^q)}}\\
\label{b}\leq &\exp(-\gamma b_T(1+\epsilon^q)^\frac{1}{q}) E\cro{\pare{1+\frac{Z_0}{s}}\exp\pare{\gamma\pare{\sum_{x\in\T_R}\frac{1}{2^q} (Z_x+s)^{2q}}^\frac{1}{q}}}.
\end{align}

Combining (\ref{ab}) and (\ref{b}), we obtain that $\forall a,\gamma,\epsilon >0$,
\begin{equation}
 \label{c}
 P\pare{I_{R,\tau}\geq b_T^q } 
\leq \exp(-\gamma b_T(1+\epsilon^q)^\frac{1}{q})\frac{E\cro{\pare{1+\frac{Z_0}{s}}\exp\pare{\frac{\gamma}{2}\|Z+s\|_{2q,R}^2}}}{P(\|Z+s\|_{2q,R}\geq \sqrt{2b_T\epsilon})}  .
\end{equation}

Let us bound $P(\|Z+s\|_{2q,R}\geq \sqrt{2b_T\epsilon})$ from below. Since
$\| Z+s\|_{2q,R} \geq \| Z\|_{2q,R}-\|s\|_{2q,R}$ and $\|s\|_{2q,R}=sR^{\frac{d}{2q}}$, we have
\begin{equation}
\label{e}
P(\|Z+s\|_{2q,R}\geq \sqrt{2b_T\epsilon})\geq P(\|Z\|_{2q,R}\geq \sqrt{2b_T\epsilon}+sR^{\frac{d}{2q}}).
\end{equation}

Then we look for an upper bound of the expectation in (\ref{c}).
Using the fact that $\forall \epsilon >0,\ (a+b)^2\leq (1+\epsilon)a^2+(1+\frac{1}{\epsilon})b^2$
and H\"older's inequality, we obtain that $\forall \epsilon >0$,
\begin{align}
&\nonumber E\cro{\pare{1+\frac{Z_0}{s}}\exp\pare{\frac{\gamma}{2}\|Z+s\|_{2q,R}^2}}\\
\nonumber\leq &E\cro{\pare{1+\frac{Z_0}{s}}\exp\pare{\frac{\gamma}{2}\pare{(1+\epsilon)\|Z\|_{2q,R}^2+(1+\frac{1}{\epsilon})s^2R^{\frac{d}{q}}}}}\\
\nonumber\leq &E \cro{\va{1+ \frac{Z_0}{s}}^{\frac{1+\epsilon}{\epsilon}}} ^{\frac{\epsilon}{1+\epsilon}} E\cro{\exp\pare{\frac{\gamma}{2} (1+\epsilon)^2 \nor{Z}_{2q,R}^2}}^{\frac{1}{1+\epsilon}} \exp\pare{\frac{\gamma}{2} \frac{1+\epsilon}{\epsilon} s^2 R^{d/q}} \\
\label{f}\leq &C(\epsilon)\pare{1+\frac{1}{s\sqrt{\lambda}}} E\cro{\exp\pare{\frac{\gamma}{2} (1+\epsilon)^2 \nor{Z}_{2q,R}^2}}^{\frac{1}{1+\epsilon}} \exp\pare{\frac{\gamma}{2} \frac{1+\epsilon}{\epsilon} s^2 R^{d/q}},
\end{align}
where the last inequality comes from the fact that $Var(Z_0)=G_{R,\lambda}(0,0)\leq E[\tau]=\frac{1}{\lambda}$.\\
We deduce from (\ref{c}), (\ref{e}) and (\ref{f}) that
$\forall \epsilon,a,\theta >0$,
\begin{align*}
&P\pare{I_{R,\tau}\geq b_T^q } \\
\leq &C(\epsilon)\exp\pare{-\gamma b_T(1+\epsilon^q)^\frac{1}{q}}\pare{1+\frac{1}{s\sqrt{\lambda}}}
\frac{E\cro{\exp\pare{\frac{\gamma}{2} (1+\epsilon)^2 \nor{Z}_{2q,R}^2}}^{\frac{1}{1+\epsilon}}}{P(\|Z\|_{2q,R}\geq \sqrt{2b_T\epsilon}+sR^{\frac{d}{2q}})} \exp\pare{\frac{\gamma}{2} \frac{1+\epsilon}{\epsilon} s^2 R^{d/q}}.
\end{align*}

The choice of s being free, we choose $s=\frac{ \epsilon\sqrt{2b_T\epsilon}}{R^\frac{d}{2q}} $. Remember that $\lambda=\frac{a b_T}{T}$ and make the change of variable $\gamma =\frac{\gamma '}{(1+\epsilon)^2}$. We have $\forall \gamma',a,\epsilon>0$,

\begin{align*}
& P\pare{I_{R,\tau}\geq b_T^q } \\
\leq & 
 C(\epsilon) \exp\pare{-\gamma' b_T\frac{(1+\epsilon^q)^\frac{1}{q}}{(1+\epsilon)^2}}
\pare{1+\frac{R^\frac{d}{2q}\sqrt{T}}{\epsilon b_T\sqrt{2a\epsilon}}} 
\frac{E\cro{\exp\pare{ \frac{\gamma '}{2} \nor{Z}_{2q,R}^2}}^{\frac{1}{1+\epsilon}}}
{P\cro{ \nor{Z}_{2q,R} \geq 2 \sqrt{2b_T\epsilon}}}
\exp\pare{\frac{\epsilon^2\gamma' b_T}{1+\epsilon}} .
\end{align*}
\end{proof}

\subsection{Step 4: Large deviations for $\|Z\|_{2q,R}$}
\begin{lemma}
\label{ldfz.lem}
Let $\tau$ and $(Z_x, x\in \T_R)$ be defined
as in theorem \ref{eisenbaum.theo}. Let $\rho_1(a,R,T)$ be defined as in Theorem \ref{fpgd}.
\begin{enumerate}
\item  $\forall a,R,T  >0$, $ 
G_{R,\lambda}(0,0) \leq \rho_1(a,R,T) \leq R^{d/q}G_{R,\lambda}(0,0)$.
\item $\forall a,\epsilon,R,T>0$, 
\begin{equation*} 
P\cro{\nor{Z}_{2q,R} \geq \sqrt{b_T \epsilon}}
\geq 
 \frac{\sqrt{\rho_1(a,R, T)}}{\sqrt{2 \pi b_T \epsilon }} 
\pare{1-\frac{\rho_1(a,R,T)}{b_T \epsilon}} 
\exp\pare{- \frac{b_T \epsilon }{2 \rho_1(a,R, T)}} .
\end{equation*} 
\item  $\exists \, C(q) \text{ such that } 
\forall a,R,T,\epsilon>0$, $\forall \gamma  \text{ such that } \gamma (1+ \epsilon) < \frac{1}{\rho_1(a,R,T)}$,
\begin{equation*} 
 E\cro{\exp\pare{\frac{\gamma}{2} \nor{Z}_{2q,R}^2}}
\leq  \frac{2}{\sqrt{1-\gamma (1+\epsilon)
\rho_1(a,R,T)}}
\exp\pare{C(q)\gamma \frac{1+\epsilon}{\epsilon} R^{d/q} G_{R,\lambda}(0,0)} .
\end{equation*} 
\end{enumerate} 
\end{lemma} 

\begin{proof}
\begin{enumerate}
\item 
For the lower bound, let take $f=\delta_0$: $\rho_1(a,R,T)\geq G_{R,\lambda}(0,0)$.\\
For the upper bound,
\begin{align*}
\rho_1(a,R,T)&=\sup\acc{\sum_{x,y\in \T_R} f_x G_{R,\lambda}(x,y) f_y \, ; \, f \text{ such that }  \nor{f}_{(2q)',R}=1}\\
&\leq \sup\limits_{x,y\in\T_R} G_{R,\lambda}(x,y) \sup\acc{ \nor{f}_{1,R}^2\, ; \, f \text{ such that }  \nor{f}_{(2q)',R}=1}.
\end{align*}
On one hand, $\nor{f}_{1,R}\leq \nor{f}_{(2q)',R}\nor{1}_{2q,R}=R^{d/2q}$. On the other hand, denote by $T_x$ the first time where the walk is at state $x$. Then, 
\begin{align*}
\sup\limits_{x,y\in\T_R}G_{R,\lambda}(x,y)&= \sup\limits_{x\in\T_R}G_{R,\lambda}(0,x)
= \sup\limits_{x\in\T_R}E_0[l_\tau^R (x)]\\
&\leq\sup\limits_{x\in\T_R}E_0[E_x[l_\tau^R(x)]\ind_{T_x\leq \tau}]
=\sup\limits_{x\in\T_R}G_{R,\lambda}(x,x)P_0(T_x\leq \tau) \\
&\leq G_{R,\lambda}(0,0).
\end{align*}
\item By H\"older's inequality, $\forall f$ such that $\|f\|_{(2q)',R}=1$
\[
P\cro{\nor{Z}_{2q,R} \geq \sqrt{b_T \epsilon}}
 \geq P \cro{\sum_{x \in \T_R} f_x Z_x \geq \sqrt{ b_T  \epsilon} }
\, .
\]
Since $\sum_{x \in \T_R} f_x Z_x$ is a real centered Gaussian variable with
variance 
\[ \sigma^2_{a,R,T}(f)= 
\sum_{x,y  \in \T_R} G_{R,\lambda}(x,y) f_x f_y \, ,
\]
we have:
\begin{eqnarray*} 
P\cro{\nor{Z}_{2q,R} \geq  \sqrt{b_T \epsilon}}
& \geq & 
\frac{\sigma_{a,R,T}(f)}{\sqrt{2\pi} \sqrt{b_T \epsilon}} 
\pare{1 - \frac{\sigma^2_{a,R,T}(f)}{b_T \epsilon}} 
\exp\pare{- \frac{b_T \epsilon }{2 \sigma^2_{a,R,T}(f)}} 
 \\
& \geq & 
\frac{\sigma_{a,R,T}(f)}{\sqrt{2\pi} \sqrt{b_T \epsilon}} 
\pare{1 - \frac{\rho_1(a,R,T)}{b_T \epsilon}} 
\exp\pare{- \frac{b_T \epsilon }{2 \sigma^2_{a,R,T}(f)}}.
\end{eqnarray*}
Taking the supremum over $f$ we obtain that $\forall a,R,T,\epsilon >0$,
$$
P\cro{\nor{Z}_{2q,R} \geq \sqrt{b_T \epsilon}}
\geq 
 \frac{\sqrt{\rho_1(a,R, T)}}{\sqrt{2 \pi b_T \epsilon }} 
\pare{1-\frac{\rho_1(a,R,T)}{b_T \epsilon}} 
\exp\pare{- \frac{b_T \epsilon }{2 \rho_1(a,R, T)}}. 
$$ 
\item
Let $M$ be the median of $\|Z\|_{2q,R}$. We can easily see that
\begin{equation}
\label{g}
E\cro{\exp\pare{\frac{\gamma}{2} \nor{Z}_{2q,R}^2}}
\leq E\cro{\exp\pare{\frac{\gamma}{2}(1+\epsilon)(\|Z\|_{2q,R}-M)^2}}
 \exp(\frac{\gamma}{2}\frac{1+\epsilon}{\epsilon}M^2).
 \end{equation}

Since $M=(\text{median}(\sum_x Z_x^{2q}))^{1/2q}$ and that for $X\geq 0,\ \text{ median}(X)\leq 2 E[X]$, we get:
\begin{align*}
M^2 &=(\text{median}(\sum_x Z_x^{2q}))^{1/q}\\
&\leq (2E[\sum_x Z_x^{2q}])^{1/q}\\
&\leq C(q)(\sum_x G_{R,\lambda}(0,0)^q E[Y^{2q}])^{1/q},\text{ where } Y\sim\mathcal{N}(0,1)\\
&\leq C(q)R^{d/q}G_{R,\lambda}(0,0)(E[Y^{2q}])^{1/q}\\
&\leq C(q)R^{d/q}G_{R,\lambda}(0,0).
\end{align*}
Thus, 
\begin{equation}
\label{h}
\exp \pare{\frac{\gamma}{2}\frac{1+\epsilon}{\epsilon}M^2}
\leq \exp \pare{\gamma\frac{1+\epsilon}{\epsilon}C(q)R^{d/q}G_{R,\lambda}(0,0)}.
\end{equation}
We find now an upper bound of the expectation in (\ref{g}).
Using concentration inequalities for norms of gaussian processes, $\forall u > 0$, 

$P\cro{\va{\nor{Z}_{2q,R} - M_{R,T}} \geq \sqrt{u}} 
\leq 2 P( Y \geq \sqrt{\frac{u}{\rho_1(a, R, T)} }) $
where $Y\sim \mathcal{N}(0,1)$. Then:
\begin{align}
\nonumber & E\cro{\exp\pare{\frac{\gamma}{2}(1+\epsilon)(\|Z\|_{2q,R}-M)^2}}\\
& \hspace*{2cm} =\nonumber 1+\int_1^{+\infty} P\pare{\exp\pare{\frac{\gamma}{2}(1+\epsilon)(\|Z\|_{2q,R}-M)^2}\geq u}du\\
& \hspace*{2cm}=\nonumber 1+\int_1^{+\infty} P\pare{\va{\|Z\|_{2q,R}-M}\geq \sqrt{\frac{2\ln(u)}{\gamma (1+\epsilon)}}}du\\
& \hspace*{2cm}\leq \nonumber 1+2\int_1^{+\infty} P\pare{Y^2\geq \frac{2\ln(u)}{\gamma (1+\epsilon)\rho_1(a,R,T)}}du\\
& \hspace*{2cm}=\nonumber -1+2  E\cro{\exp\pare{\frac{\gamma(1+\epsilon)\rho_1(a, R,T)}{2 
} Y^2}}
\\
& \hspace*{2cm}=  -1 + \frac{2}{\sqrt{1-\gamma(1+\epsilon)\rho_1(a,R,T)}}
\leq  \label{i}\frac{2}{\sqrt{1-\gamma(1+\epsilon)\rho_1(a,R,T)}}.
\end{align}
Remark that it is only true for $\gamma,\epsilon$ such that $\gamma (1+\epsilon)<\frac{1}{\rho_1(a,R,T)}$.
We deduce putting together (\ref{g}),(\ref{h}) and (\ref{i}), that
$$E\cro{\exp\pare{\frac{\gamma}{2} \nor{Z}_{2q,R}^2}}
\leq \frac{2}{\sqrt{1-\gamma(1+\epsilon)\rho_1(a,R,T)}}\exp \pare{\gamma\frac{1+\epsilon}{\epsilon}C(q)R^{d/q}G_{R,\lambda}(0,0)}  .$$

\end{enumerate}
\end{proof}

\subsection{Proof of Theorem \ref{fpgd}}

\begin{proof}First we remark that if $\rho_1$ is infinite, then theorem \ref{fpgd} is obvious. So we assume now that $\rho_1$ is finite.
Combining Lemma \ref{majoesp} and Lemma \ref{firstbornesup} we have proved that: $\forall \epsilon,\gamma,a,R,T>0$,
\begin{equation}
 P \cro{ I_T\geq b_T^q} 
\leq \label{j}C(\epsilon)\exp(a b_T)\exp\pare{-\gamma b_T (1+\circ(\epsilon))}
\pare{1+\frac{R^\frac{d}{2q}\sqrt{T}}{\epsilon b_T\sqrt{2a\epsilon}}} 
\frac{E\cro{\exp\pare{ \frac{\gamma }{2} \nor{Z}_{2q,R}^2}}^{\frac{1}{1+\epsilon}}}
{P\cro{ \nor{Z}_{2q,R} \geq 2 \sqrt{2b_T\epsilon}}}.
\end{equation}
First, lemma \ref{ldfz.lem} gives that $\forall \gamma$ such that $\gamma(1+\epsilon)<\frac{1}{\rho_1(a,R,T)}$,\\
$E \cro{\exp\pare{\frac{\gamma}{2} \nor{Z}_{2q,R}^2}}^{\frac{1}{1+\epsilon}} \leq \exp (\frac{\gamma}{\epsilon}C(q)R^{d/q}G_{R,\lambda}(0,0)) \pare{\frac{2}{\sqrt{1-\gamma(1+\epsilon)\rho_1(a,R,T)}}}^{\frac{1}{1+\epsilon}}.$\\
Since $\rho_1$ is finite, for $a$ little enough, $1/\rho_1(a)>0$ and we can choose $\gamma$ such that $0<\gamma<\frac{1}{\rho_1(a)}$.
 Then it is possible to choose $\epsilon>0$ such that $\gamma (1+2\epsilon)<\frac{1}{\rho_1(a)}$. Hence
 for T sufficiently large $\frac{1}{\rho_1(a,R,T)}>\gamma (1+2\epsilon)$, then it follows that 
\\$E \cro{\exp\pare{\frac{\gamma}{2} \nor{Z}_{2q,R}^2}}^{\frac{1}{1+\epsilon}} \leq \exp (\frac{\gamma}{\epsilon}C(q)R^{d/q}G_{R,\lambda}(0,0)) \pare{2\sqrt{\frac{1+2\epsilon}{\epsilon}}}^\frac{1}{1+\epsilon}$.\\
 We recall that we have assumed that $\lambda$ and $R$ depend on $T$ in such a way that $\lambda R^d \gg1$ and $\lambda\ll 1$, which implies that we are in conditions  of application of Lemma \ref{green}. So we know that $G_{R,\lambda (0,0)}\rightarrow G(0,0)$. Moreover we have assumed that $b_T\gg R^\frac{d}{q}$, therefore we have:
\begin{equation}
\label{k}
 \limsup_{T \rightarrow \infty} \frac{1}{b_T} \log E \cro{\exp\pare{\frac{\gamma}{2} \nor{Z}_{2q,R}^2}}^{\frac{1}{1+\epsilon}} =0.
 \end{equation}
Then we work on the probability $P \cro{ \nor{Z}_{2q,R} \geq \sqrt{8 b_T \epsilon}}$ in (\ref{j}).\\
 In the same way that previously we use $\rho_1(a,R,T)<\frac{1}{\gamma (1+2\epsilon)}$,
 $\rho_1(a,R,T)\geq G_{R,\lambda}(0,0)$ and lemma \ref{ldfz.lem} to obtain:
\begin{align*}
P \cro{ \nor{Z}_{2q,R} \geq \sqrt{8 b_T \epsilon}}
\geq &\frac{\sqrt{\rho_1(a,R, T)}}{4\sqrt{ \pi b_T \epsilon }} 
\pare{1-\frac{\rho_1(a,R,T)}{8b_T \epsilon}} 
\exp\pare{- \frac{4b_T \epsilon }{ \rho_1(a,R, T)}}\\
\geq &\frac{\sqrt{G_{R,\lambda}(0,0)}}{4\sqrt{ \pi b_T \epsilon }} 
\pare{1-\frac{1}{8b_T \epsilon\gamma (1+2\epsilon)}} 
\exp\pare{-\frac{4b_T \epsilon }{ G_{R,\lambda}(0,0)}}.
\end{align*}
We conclude from $G_{R,\lambda (0,0)}\rightarrow G(0,0)$ that
\begin{equation}
\label{l}
\limsup_{T \rightarrow \infty} \frac{1}{b_T} \log  P \cro{ \nor{Z}_{2q,R} \geq \sqrt{8 b_T \epsilon}} \geq -\frac{4\epsilon}{ G(0,0) }   . 
\end{equation}
Putting together (\ref{j}),(\ref{k}) and (\ref{l}), we have for $b_T\gg \log(T)$
\[
\limsup_{T \rightarrow \infty} \frac{1}{b_T} \log 
P\cro{I_T\geq b_T^q}
\leq a   -\gamma (1+\circ (\epsilon))+\frac{ 4 \epsilon}{  G(0,0)} \, .
\]

Let send $\epsilon$ to 0 then $\gamma$ to $\frac{1}{\rho_1(a)}$. We obtain that
for $a$ little enough
$$\limsup_{T \rightarrow \infty} \frac{1}{b_T} \log 
P\cro{I_T\geq b_T^q}
\leq a -\frac{1}{\rho_1(a)}.$$

Let $(a_n)$ be a sequence converging to 0 such that 
$\limsup\limits_{n \rightarrow \infty} \rho_1(a_n) = \rho_1$:
\[\limsup_{T \rightarrow \infty} 
\frac{1}{b_T} \log 
P\cro{I_T\geq b_T^q}
\leq a_n -\frac{1}{\rho_1(a_n)}\, .
\] 
Then we let $n$ go to infinity. We finish the proof by showing that the conditions $\lambda R^d\gg 1$, $b_T^q\gg R^d$ and $\log(T)\ll b_T\ll T$ are compatible. Indeed, the first two conditions imply that $b_T\gg T^\frac{1}{q+1}$. In conclusion, we have proved that for $T^\frac{1}{q+1}\ll b_T\ll T$:
$$\limsup\limits_{T\rightarrow +\infty}\frac{1}{b_T} \log P\cro{I_T\geq b_T^q}\leq -\frac{1}{\rho_1}.$$
\end{proof}
\section{Lower bound}

This part is devoted to the proof of the large deviations lower bound.
\begin{theo}
\label{tlb}
  Lower bound for $I_T$. \\
 Assume that $q(d-\alpha) \geq d$ and $b_T\ll T$ then
\begin{equation} 
\label{SILT4LB.eq}
 \liminf_{T \rightarrow \infty} \frac{1}{b_T} \log P \cro{I_T \geq b_T^q}
\geq  - \kappa(q)  .
\end{equation}
\end{theo}

\begin{proof}
Fix $M > 0$. 
Let $T_0$ be such that for all $T \geq T_0$, $\frac{T}{b_T} > M$. For 
$T \geq T_0$, we have:
\[
P\cro{I_T \geq b_T^q } 
\geq P\cro{I_{Mb_T} \geq b_T^q }
= P\cro{\nor{\frac{l_{Mb_T}}{Mb_T}}_q \geq \frac{1}{M}}.
\]
The function $\nu \in \FF \mapsto \nor{\nu}_q
=\sup\limits_{f; \nor{f}_{q'}=1} \acc{\sum_x \nu(x) f(x)}$ is lower semicontinuous in $\tau$-topology hence
$\forall t >0$, $\acc{ \nu \in \FF\, ,  \nor{\nu}_q >t}$ is an open
subset of $\FF$. Therefore, using the classical results of Donsker and Varadhan \cite{DV} on local time of Markov process, we have that $\forall \epsilon >0$,
\begin{eqnarray*} 
\liminf_{T \rightarrow  \infty} \frac{1}{Mb_T} \log 
P \cro{\nor{\frac{l_{Mb_T}}{Mb_T}}_q \geq \frac{1}{M}} 
& \geq & \liminf_{T \rightarrow  \infty}
 \frac{1}{Mb_T} \log 
P \cro{\nor{\frac{l_{Mb_T}}{Mb_T}}_q >  \frac{1-\epsilon}{M}}
\\
& \geq & - \inf\limits_f \acc{<f,-Af> \, ; \nor{f}_2=1 \, , \nor{f}_{2q}^2 > 
 \frac{1-\epsilon}{M}} \, . 
\end{eqnarray*} 

We have thus proved that  $\forall M >0$, 
$\forall \epsilon >0$,
\[\liminf_{T \rightarrow  \infty} \frac{1}{b_T} \log
P \cro{I_T \geq  b_T^q} \geq -  M \kappa_1\pare{\frac{1-\epsilon}{M}} \, 
\]
where $\kappa_1(y) := \inf\limits_f \acc{<f,-Af> \, ; \,\,  \nor{f}^2_{2q} >  
y  \, , \,\, \nor{f}_2=1}$.\\
It remains to prove that for
$\forall y > 0$,\ \ \ $\inf_{M > 0} M \kappa_1(y/M) =  y\kappa (q)$.
 \begin{eqnarray*} 
\inf_{M > 0} M \kappa_1(y/M) 
& = & y \inf_{M > 0} M \kappa_1(1/M)
\\ 
& = & y \inf_{M > 0} \inf_{f} \acc{ M <f,-Af> \, ; 
	\nor{f}_2=1 \, , \nor{f}_{2q}^2 > \frac{1}{M} }
\\
& = &  y \inf_f \inf_{M > 0} \acc{M <f,-Af>;
M > \frac{1}{  \nor{f}_{2q}^2},\ \nor{f}_2=1}
\\
\label{ro3.eq}
& = & y \inf_f \acc{ 
\frac{<f,-Af>}{\nor{f}_{2q}^2}, \ \nor{f}_2=1} \, ;
\\
& = & y\kappa(q) \, .
\end{eqnarray*} 
To finish the proof it suffices to let $\epsilon\rightarrow 0$.

\end{proof}
\section{Proof of proposition \ref{kapparho} and theorem \ref{ld}}
Until now we have obtained a lower bound with $\kappa(q)$ and an upper bound with $\rho_1$. We show in proposition \ref{rho1} another upper bound for large deviations of $I_T$ with the constant $\rho(q)$. Then in proposition \ref{kappafini} we prove that $\kappa(q)$ is a non degenerate constant and we finish the proof of our large deviations principle with Proposition \ref{kappa}, where we show that the upper bound and the lower bound are the same.

\begin{prop}: {\bf Behavior of $\rho_1(a, R, T)$}.\\
\label{rho1}
Assume that $q(d-\alpha)\geq d$ and that
$\lambda$ and $R$ depend on $T$ in such a way that $\lambda R^{d/q'} \gg 1$, then
under assumption 1 we have: $$\rho_1 \leq \rho (q).$$
\end{prop}

\begin{proof}
By definition $ \rho_1(a,R,T) 
= \sup\limits_{f}\acc {\sum\limits_{x,y \in \T_R }f(x)G_{R,\lambda}(x-y)f(y)\, ; \nor{f}_{(2q)',R}=1}$.\\
Since the space of $\{f/\nor{f}_{(2q)',R} = 1\}$ is compact there exists $f_0\ \in l_{(2q)'}(\T_R)$
realizing the supremum. Of course $f_0\geq 0$ since the supremum is obtained with  non-negative function.\\
Let $0<r<R$ and define
\[
\CC_{r,R} = \cup_{i=1}^d
\acc{ x\in \Z^d \, ; 0 \leq x_i \leq r \mbox{ or } R-r \leq x_i \leq R} \, .
\]
We can assume that $\sum\limits_{x\in\CC_{r,R}} f_0(x)^{(2q)'} \leq \frac{2dr}{R}$.
Indeed on one side we have
\begin{align*}
&\sum\limits_{a \in [0,R]^d}  \sum\limits_{x \in \CC_{r,R}}  f_0(x-a)^{(2q)'}  
=  \sum\limits_{x \in \CC_{r,R}} \sum\limits_{a \in [0,R]^d}  f_0(x-a)^{(2q)'} 
\\
& =  \sum\limits_{x \in \CC_{r,R}} \sum\limits_{x \in \T_R} f_0(x)^{(2q)'} 
= \card(\CC_{r,R}) \nor{f_0}_{(2q)'}^{(2q)'}
\leq 2d r R^{d-1},
\end{align*}
and on the opposite side we have
\[ \sum_{a \in [0,R]^d} \sum_{x \in \CC_{r,R}} f_0(x-a)^{(2q)'} 
\geq R^d \inf_{a \in [0;R]^d} \sum_{x \in \CC_{r,R}} f_0(x-a)^{(2q)'} \, .
\]
Thus \[\inf_{a \in [0;R]^d}\{\sum_{x \in \CC_{r,R}} f_0(x-a)^{(2q)'}\}   \leq\frac{2dr}{R} .\]
Moreover $f_{0,a}(x):= f_0(x-a)$
 is a periodic function of period $R$. Note that 
$\nor{f_{0,a}}_{(2q)',R} = \nor{f_0}_{(2q)',R}$ and  
$\sum\limits_{x,y \in \T_R }f_{0,a}(x)G_{R,\lambda}(x-y)f_{0,a}(y)=\sum\limits_{x,y \in \T_R }f_0(x)G_{R,\lambda}(x-y)f_0(y)$. \\
Finally, we can assume that 
\begin{equation}
\label{m}
\sum\limits_{x\in\CC_{r,R}} f_0(x)^{(2q)'} \leq \frac{2dr}{R}.
\end{equation}
Let $\psi: \Z^d \mapsto [0,1]$ be a truncature function satisfying 
\[ \left\{ \begin{array}{ll}
	\psi(x) = 0 & \mbox{ if } x \notin [0;R]^d \, 
	\\
	\psi(x) = 1 & \mbox{if } x \in [0;R]^d \backslash \CC_{r,R} \, .
	\end{array} 
\right.
\] 
Let $g_{0}=\frac{\psi f_0}{\nor{\psi f_0}_{(2q)'}}$ be our candidate to realize the supremum in the definition of $\rho(q)$. Fix $\epsilon \in ]0,1[$ and take $r=\frac{\epsilon R}{2d}$. First we can remark that $\nor{\psi f_0}_{(2q)'}>0$. Indeed:
\begin{equation*} 
\nor{\psi f_0}^{(2q)'}_{(2q)'} \geq \sum_{x \in [0;R]^d} f_0^{(2q)'}(x) - 
 \sum_{x \in \CC_{r,R}} f_0^{(2q)'}(x) 
 \geq  1- \frac{2dr}{R}= 1- \epsilon >0.   
\end{equation*} 

 By Lemma \ref{ineqgreen}, there exists a constant $C$ such that $\forall\lambda ,R>0$, $G(x)\geq G_{R,\lambda}(x)-\frac{C}{\lambda R^d}$, hence:
\begin{align}
&\nonumber\sum\limits_{x,y \in \mathbb{Z}^d} g_{0}(x)G(x-y)g_0(y) \\
=& \nonumber\frac{1}{\nor{\psi f_0}_{(2q)'}^2} \sum\limits_{x,y \in \mathbb{Z}^d }\psi(x)f_0(x)G(x-y)\psi(y)f_0(y)\\
\geq & \nonumber \sum\limits_{x,y \in \mathbb{Z}^d }\psi(x)f_0(x)G(x-y)\psi(y)f_0(y)\\
\geq & \nonumber \sum\limits_{x,y \in [0,R]^d}f_0(x)G(x-y)f_0(y)-2\sum\limits_{x\in [0,R]^d,y \in \CC_{r,R} }f_0(x)G(x-y)f_0(y)\\
=&\label{n}\rho_1(a,R,T)-\frac{C}{\lambda R^d}\pare{\sum\limits_{x \in [0,R]^d}f_0(x)}^2-2\sum\limits_{x\in [0,R]^d,y \in \CC_{r,R}}f_0(x)G(x-y)f_0(y).
\end{align}
Let us work on (\ref{n}). We first show that $\sum\limits_{x \in [0,R]^d }f_0(x) \leq R^\frac{d}{2q}$:
\begin{eqnarray}
\label{p}
\sum\limits_{x\in [0,R]^d  }f_0(x) &\leq &\pare{\sum\limits_{x\in [0,R]^d  }f_0^{(2q)'}(x)}^\frac{1}{(2q)'}  (R^d) ^\frac{1}{2q}
= R^\frac{d}{2q}.
\end{eqnarray}

We control now $\sum\limits_{x\in [0,R]^d,y \in \CC_{r,R} }f_0(x)G(x-y)f_0(y) $. Using (\ref{m}) and the fact that $\nor{f_0}_{(2q)',R}=1$ we have:

\begin{align}
&\nonumber\sum\limits_{x\in [0,R]^d,y \in \CC_{r,R} }f_0(x)G(x-y)f_0(y) \\
\nonumber = &\sum\limits_{x\in [0,R]^d,y \in \CC_{r,R}}f_0^{1/(2q-1)}(x)f_0^{1/(2q-1)}(y)G(x-y)f_0^{\frac{2(q-1)}{2q-1}}(x)f_0^{\frac{2(q-1)}{2q-1}}(y)\\
\leq &\nonumber  \pare{\sum\limits_{x\in [0,R]^d,y \in \CC_{r,R}}f_0^{q/(2q-1)}(x)f_0^{q/(2q-1)}(y)G^q(x-y)}^{1/q}  \pare{\sum\limits_{x\in [0,R]^d,y \in \CC_{r,R}}f_0^{\frac{2q}{2q-1}}(x)f_0^{\frac{2q}{2q-1}}(y)}^{(q-1)/q}\\
 \leq &\nonumber \pare{\sum\limits_{z\in[-R,R]^d}G^q(z)\sum\limits_{y\in\CC_{r,R}}f_0^{q/(2q-1)}(z+y)f_0^{q/(2q-1)}(y) }^{1/q} \pare{\sum\limits_{x\in [0,R]^d}f_0^{\frac{2q}{2q-1}}(x)}^\frac{q-1}{q}\\
 &\nonumber \pare{\sum\limits_{x\in \CC_{r,R}}f_0^{\frac{2q}{2q-1}}(x)}^\frac{q-1}{q}\\
 \leq &\nonumber\epsilon^\frac{q-1}{q}\pare{\sum\limits_{z\in[-R,R]^d}G^q(z)\pare{\sum_{y\in\CC_{r,R}}f_0^{\frac{2q}{2q-1}}(y)}^{1/2}\pare{\sum_{y\in\CC_{r,R}}f_0^{\frac{2q}{2q-1}}(z+y)}^{1/2}}^{1/q}\\
 \leq &\label{q}\epsilon^\frac{2q-1}{2q}\pare{\sum\limits_{z\in [-R,R]^d}G^q(z)}^{1/q}.
\end{align}

Finally, putting together (\ref{n}),(\ref{p}) and (\ref{q}), we deduce that:
\begin{equation*}
\sum\limits_{x,y \in \mathbb{Z}^d} g_0(x)G(x-y)g_0(y) 
\geq \rho_1(a,R,T) - R^\frac{d}{q}\frac{C}{\lambda R^d} -2\epsilon^\frac{2q-1}{2q}\pare{\sum\limits_{z\in [-R,R]^d}G^q(z)}^{1/q}  .
\end{equation*}
Let $\epsilon\rightarrow 0$: $\sum\limits_{x,y \in \mathbb{Z}^d} g_0(x)G(x-y)g_0(y) 
\geq \rho_1(a,R,T) -  \frac{C}{\lambda R^{d/q'}}.$

Hence, \[ \sup\limits_{g} \acc{\sum\limits_{x,y \in \mathbb{Z}^d} g(x)G(x-y)g(y) ,\ \nor{g}_{(2q)'}=1, \text{ supp}(g)\subset [0,R]^d}\geq \rho_1(a,R,T) - \frac{C}{\lambda R^{d/q'}}.\]
Therefore,  \[ \sup\limits_{g} \acc{\sum\limits_{x,y \in \mathbb{Z}^d} g(x)G(x-y)g(y) ,\ \nor{g}_{(2q)'}=1, \text{ supp}(g)\ \text{compact}}\geq \rho_1(a,R,T) -  \frac{C}{\lambda R^{d/q'}}.\]

Then we take a sequence $T_n\rightarrow +\infty$ such that $\rho_1(a,R,T_n)\rightarrow \rho_1(a)$. Hence by definition of $\rho(q)$ we obtain: 
 \[ \rho(q)\geq \rho_1(a) .\]
Then we take a sequence $a_n\rightarrow 0$ such that $\rho_1(a_n)\rightarrow \rho_1$. Hence, 
\[ \rho(q)\geq \rho_1 .\]
\end{proof}

\begin{prop}\label{kappafini}
 Under assumption 1,
\begin{enumerate}
\item If $q(d-\alpha)>d$ then $0<\rho(q)<+\infty$.
\item  If $q(d-\alpha)=d$ then $0<\kappa (q)<+\infty$.
\end{enumerate}
\end{prop}

\begin{proof}
\begin{enumerate}
\item It is easy to see that $\rho (q)>0$. Indeed, taking $f=\delta_0$ gives $\rho (q)\geq G_d(0,0)$. 
Now we show that $\rho(q)$ is finite. We proceed in the same way that in proposition \ref{rho1}, for all $f$ with compact support such that $\nor{f}_{(2q)'}=1$,
\begin{equation}
\begin{array}{l}
\sum\limits_{x \in \mathbb{Z}^d }f(x)G(x-y)f(y) 
= \sum\limits_{x,y\in\mathbb{Z}^d}f^{1/(2q-1)}(x)f^{1/(2q-1)}(y)G(y-x)f^{\frac{2(q-1)}{2q-1}}(x)f^{\frac{2(q-1)}{2q-1}}(y)\\
 \leq  \pare{\sum\limits_{x,y\mathbb{Z}^d}f^{q/(2q-1)}(x)f^{q/(2q-1)}(y)G^q(y-x)}^{1/q}  \pare{\sum\limits_{x,y\in \mathbb{Z}^d}f^{\frac{2q}{2q-1}}(x)f^{\frac{2q}{2q-1}}(y)}^{(q-1)/q}\\
 \leq  \pare{\sum\limits_{x\in\mathbb{Z}^d}G^q(x)\sum\limits_{y\in\mathbb{Z}^d}f^{q/(2q-1)}(x+y)f^{q/(2q-1)}(y) }^{1/q}  \pare{\sum\limits_{x\in \mathbb{Z}^d}f^{\frac{2q}{2q-1}}(x)}^\frac{2(q-1)}{q}\\
 \leq \pare{\sum\limits_{x\in\mathbb{Z}^d}G^q(x)\pare{\sum_{y\in\mathbb{Z}^d}f^{\frac{2q}{2q-1}}(y)}^{1/2}\pare{\sum\limits_{y\in \mathbb{Z}^d}f^{\frac{2q}{2q-1}}(x+y)}^{1/2}}^{1/q}\\
 =\pare{\sum\limits_{x\in [0,R]^d}G^q(x)}^{1/q}=\nor{G}_q.
\end{array}
\end{equation}
Then we take the supremum over $f$. Moreover, thanks to the work of Le Gall and Rosen \cite{LeGallRosen}, we know that $G(0,x)=O (\vert x\vert^{\alpha-d})$. Then $\nor{G}_q$ is finite since $q(d-\alpha)>d$.

\item
To prove $\kappa(q)<\infty$ it suffices to take $f=\delta_0$. Indeed $\kappa(q)\leq <-A\delta_0,\delta_0>=1-\mu(0)<+\infty$. Let us now prove that $\kappa (q)>0$. The solution comes from the following result due to Varopoulos in \cite{Varopoulos}:

Let $\nu >2$. If $p_t$ is the transition probability of a symmetric Markov process $(Y_t,t\geq 0)$ defined on a measure space $X$, with $V$ is the domain of the generator of $(Y_t,t\geq 0)$ and $\mathcal{E}$ its Dirichlet form. Then the following assertions are equivalent:
\begin{enumerate}
\item 
$\exists C>0 \text{ such that } \forall x,y\in\mathbb{Z}^d, p_t(x,y)\leq \frac{C}{t^{\nu/2}} .$
\item $\exists C'>0$ such that $\forall f\in\mathcal{K}\cap V$,$\nor{f}_{\frac{2\nu}{\nu-2}}^2 \leq C' \mathcal{E} (f,f),$\\
where $\mathcal{K}=\acc{f\in L^\infty (X), \text{ supp}(f) \text{ compact}}$.
\end{enumerate}
By Proposition 4.2 in \cite{BassLevin} due to Bass and Levin, we know that 
$$\exists C>0 \text{ such that } \forall x,y\in\mathbb{Z}^d,\ p_t(x,y)\leq C t^{-\frac{d}{\alpha}}.$$
Since $q(d-\alpha)=d$, $\nu=\frac{2d}{\alpha}>2$. So, there exists
 $C'>0$ such that $\forall  f\in\mathcal{K}\cap V$, $\nor{f}_{2q}^2=\nor{f}_{\frac{2d}{d-\alpha}}^2\leq C' \mathcal{E}(f,f)$. 
 Let $f$  with compact support such that $\nor{f}_2=1$. Of course $f\in\mathcal{K}$. If $f\in V$ then
$\nor{f}_{\frac{2d}{d-\alpha}}^2\leq C' \mathcal{E}(f,f)$. If $f\not\in V$ then $\mathcal{E}(f,f)=+\infty$ and the inequality is also true. Thus,
\[\forall f \text{ with compact support such that }\nor{f}_2=1,\nor{f}_{2q}^2\leq C' \mathcal{E}(f,f) .\]
Therefore, taking the infimum over all function $f$ such that $\nor{f}_2=1$ we have: 
$$\inf\limits_{f}\acc{\frac{\mathcal{E}(f,f)}{\nor{f}_{2q}^2},\nor{f}_2=1}=\inf\limits_{f}\acc{\frac{\mathcal{E}(f,f)}{\nor{f}_{2q}^2},\nor{f}_2=1,\text{ supp(f) \text{ compact}}}\geq \frac{1}{C'}$$
\end{enumerate}
\end{proof}

\begin{prop}
\label{kappa} Under assumption 1, if $q(d-\alpha)\geq d$ then
 $\kappa (q)=\frac{1}{\rho(q)}$. 
\end{prop}
\begin{proof}
By theorem \ref{fpgd}, theorem \ref{tlb} and proposition \ref{rho1} we know that $\frac{1}{\rho(q)}\leq \kappa (q)$. So we just have to prove that $\kappa (q) \leq \frac{1}{\rho (q)}$.\\
By definition $\rho (q) =  \sup\limits_{g} \acc {<g,Gg>, \text{ supp}(g) \text{ compact},\ \nor{g}_{(2q)'}=1}$. Note that 
\begin{equation}
\label{az}
\rho (q) =  \sup\limits_{g} \acc {<g,Gg>, \text{ supp}(g) \text{ compact},\ \nor{g}_{(2q)'}=1, \ \nor {Gg}_2<+\infty}.
\end{equation}
Indeed if $g$ has compact support and $\nor{g}_{(2q)'}=1$ then $\nor{Gg}_2<+\infty$.\\
We have seen in proposition \ref{kappafini} that $\rho(q)>0$ when $q(d-\alpha)>d$ but the proof is also true when $q(d-\alpha)=d$. Furthermore proposition \ref{kappafini} gives us that if $q(d-\alpha)>d$ then $\rho(q)$ is finite. We proceed by contradiction to see that it is also true when $q(d-\alpha)=d$ using the same method that Chen and M\"{o}rters in \cite{ChenMorters}.

Assume that $\rho(q)=+\infty$. Then by (\ref{az}), $\forall B>0$
there exists $g$ with compact support, $\nor{g}_{(2q)'}=1$ 	and $\nor{Gg}_2<+\infty$ such that $<g,Gg>\ \geq B$.\\
Note that $<g,Gg> \leq \nor{g}_{(2q)'} \nor{Gg}_{2q}=\nor{Gg}_{2q}$. So $\nor{Gg}_{2q}\geq B$.\\
Then we set $f=\frac{Gg}{\nor{Gg}_{2q}}$. We note that $\nor{f}_{2q}= 1$ and $\nor{f}_2<+\infty$, hence:
\begin{align}
\nonumber <g,Gg> &= <-AGg,Gg> \\
\nonumber              &= \nor{Gg}_{2q}^2 <-\frac{AGg}{\nor{Gg}_{2q}}, \frac{Gg}{\nor{Gg}_{2q}}>\\
\nonumber              &\geq  \nor{Gg}_{2q}^2\inf\limits_f \acc{<-Af,f>, \ \nor{f}_{2q}= 1,\ \nor{f}_2<+\infty}\\
\nonumber              &= \nor{Gg}_{2q}^2\inf\limits_f \acc{\frac{<-Af,f>}{\nor{f}_2^2} \nor{f}_2^2, \ \nor{f}_{2q}=1,\      \nor{f}_2<+\infty}\\
\label{pw}&=  \nor{Gg}_{2q}^2\inf\limits_g \acc{\frac{<-Ag,g>}{\nor{g}_{2q}^2}, \ \nor{g}_2= 1}
=  \nor{Gg}_{2q}^2\kappa (q)
\end{align}
with $g=\frac{f}{\nor{f}_2}$. Therefore,
 $$\kappa(q)\leq \frac{<g,Gg>}{ \nor{Gg}_{2q}^2}\leq \frac{1}{\nor{Gg}_{2q}}\leq \frac{1}{B}$$
 then letting $B\rightarrow +\infty$ we have $\kappa(q)=0$, which is in contradiction with proposition \ref{kappafini}. Therefore $\rho(q)$ is finite.

Now we proceed in the same way that previously.  Let $\epsilon \in]0,\rho(q)[$, by (\ref{az})
there exists $g$ with compact support, $\nor{g}_{(2q)'}=1$ 	and $\nor{Gg}_2<+\infty$ such that $\rho (q) \geq\ <g,Gg>\ \geq \rho (q)-\epsilon$.
Moreover we have $\nor{Gg}_{2q}\geq \rho (q)-\epsilon$,
then we set $f=\frac{Gg}{\rho (q)-\epsilon}$ and obtain
\begin{align*}
\rho (q) \geq\ <g,Gg>&\geq  (\rho (q)-\epsilon)^2 \inf\limits_f \acc{<-Af,f>, \ \nor{f}_{2q}\geq 1,\ \nor{f}_2<+\infty}\\
&= (\rho (q)-\epsilon)^2 \inf\limits_f \acc{<-Af,f>, \ \nor{f}_{2q}= 1,\ \nor{f}_2<+\infty}.
\end{align*}

Let $\epsilon \rightarrow 0$: $\frac{1}{\rho (q)}\geq \inf\limits_f \acc{<-Af,f>, \ \nor{f}_{2q}= 1,\ \nor{f}_2<+\infty}$.\\
Moreover we have seen in (\ref{pw}) that 
$\inf\limits_f \acc{<-Af,f>, \ \nor{f}_{2q}\geq 1,\ \nor{f}_2<+\infty} = \kappa (q)$, therefore $\kappa(q)\leq \frac{1}{\rho(q)}.$

\end{proof}

\textsc{Cl\'ement Laurent\\
LATP, UMR CNRS 6632\\
CMI, Universit\'e de Provence\\
39 rue Joliot-Curie, F-13453 Marseille cedex 13, France\\
laurent@cmi.univ-mrs.fr  }


\begin{thebibliography}{20}
\bibitem{Asselah5} Asselah.A, Large deviation principle for sel-intersection local times for random walk in $\mathbb{Z^d}$ with $d\geq 5$.  ALEA Lat. Am. J. Probab. Math. Stat.  6  (2009), 281--322. 60F10 (60G50 60K35)
\bibitem{BCR2} Bass.R.F, Chen.X, Rosen.J, Moderate deviations and laws of the iterated logarithm for the renormalized self-intersection local times of planar random walks.  Electron. J. Probab.  11  (2006), no. 37, 993--1030 
\bibitem{BCR} Bass.R.F, Chen.X, Rosen.J, Large deviations for renormalized self-intersection local times of stable processes, The Annals of Probability, 2005, Vol 33, No 3, 984-1013, DOI 10.1214/009117904000001099
\bibitem{BassLevin} Bass.R.F, Levin.D.A, Transition probabilities for symmetric jump process, Transaction of the AMS, Vol 354, no 7, Pages 2933-2953, S 0002-9947(02)02998-7
\bibitem{Castell} Castell.F; Large deviations for intersection local times in critical dimension, arXiv:0812.1639
\bibitem{BassLi} Chen.X, Li.W.V, Large and moderate deviations for intersection local times.  Probab. Theory Related Fields  128  (2004),  no. 2, 213--254.
\bibitem{Chen} Chen.X, Random walk intersections: large deviations and related topics, soon published by Mathematical Survey and Monorgraphy, American Mathematical Society
\bibitem{ChenLiRosen} Chen.X, Li.W.V, Rosen.J, Large Deviations for Local Times of Stable Processes and Stable Random Walks in 1 dimension, Electronic Journal of Probability, Vol.10(2005), No 16, pages 577-608
\bibitem{ChenRosen} Chen.X, Rosen.J, Exponential asymptotics for intersection local times of stable processes and random walks, Ann I.H. Poincarre-PR 41 (2005) 901-928
\bibitem{ChenMorters} Chen.X, M\"{o}rters.P, Upper tails for intersection local times of random walks in supercritical dimensions, Oberwolfach Preprints (2008 )ISSN 1864-7596
\bibitem{DV} Donsker.M.D, Varadhan.S.R.S, Asymptotic evaluation of certain Markov process expectations for large time I. Comm. Pure. Appl. Math. 28 (1975), 389-461.
\bibitem{LeGallRosen} Le Gall.J.F, Rosen.J, The range of stable random walks, The Ann of Proba, 1991, Vol 19, No 2, 650-705
\bibitem{MarcusRosen} Marcus.M.B, Rosen.J, Markov Processes, Gaussian Processes, and Local Times, Cambridge studies advanced mathematics 100, Cambridge University Press, 2006
\bibitem{Varopoulos} Varopoulos.N.Th, Hardy-Littlewood theory for semigroups, Journal of functional analysis, 1985, Vol. 63, no 2, Pages 240-260, 0022-1236 
\bibitem{West} Westwater.J, On Edwards' model for long polymer chains. Commun. Math. Phys., 72:131-174, (1980).

\end{thebibliography}
\end{document}